\documentclass[a4paper,11pt,reqno]{amsart}
\usepackage{amsmath}
\usepackage[T1]{fontenc}
\usepackage{graphicx}
\usepackage{amssymb}
\usepackage{amsthm}
\usepackage{color}
\usepackage[lmargin=2.5 cm,rmargin=2.5 cm,tmargin=3.5cm,bmargin=2.5cm,
paper=a4paper]{geometry}

\newcommand{\md}{\,{\rm d}}
\newcommand{\Ab}{\mathbf A}

\newcommand{\R}{\mathbb R}
\newcommand{\C}{\mathbb C}
\newcommand{\E}{\mathrm{E}_{\rm gs}(\kappa,H)}

\DeclareMathOperator{\curl}{curl}
 \DeclareMathOperator{\dist}{dist}

\newtheorem{thm}{Theorem}[section]
\newtheorem{prop}[thm]{Proposition}
\newtheorem{lem}[thm]{Lemma}

\theoremstyle{remark}
\newtheorem{rem}[thm]{Remark}

\numberwithin{equation}{section}

\title[GL order parameter]{The Ginzburg-Landau order parameter near the second critical field}

\author{Ayman Kachmar}
\address[A. Kachmar]{Department of Mathematics, Lebanese University, Hadat, Lebanon; School of Arts and Sciences, Lebanese International University, Beirut, Lebanon}
\email{ayman.kashmar@liu.edu.lb}
\date{\today}

\begin{document}

\begin{abstract}
In Ginzburg-Landau theory of superconductivity, the density  and
location of the superconducting electrons are measured by a
complex-valued wave function,  the order parameter. In this paper,
when  the intensity of the applied magnetic field is close to the
second critical field, and when the order parameter minimizes the
Ginzburg-Landau functional defined over a two dimensional domain,
the  leading order approximation of its $L^2$-norm  in `small'
squares is given as the Ginzburg-Landau parameter tends to infinity.
\end{abstract}

\maketitle 

\section{Introduction}\label{hc2-sec:int}

In this paper, we study the minimizers of the Ginzburg-Landau
functional of superconductivity. In a two bounded and dimensional simply
connected domain $\Omega$ with smooth boundary, the functional is
defined over configurations $(\psi,\Ab)\in H^1(\Omega;\mathbb
C)\times H^1(\Omega;\mathbb R^2)$ as follows,
\begin{align}\label{eq-3D-GLf}
\mathcal E(\psi,\Ab)&=\int_\Omega e_{\kappa,H}(\psi,\Ab)\,dx\nonumber\\
&=\int_\Omega \left(|(\nabla-i\kappa
H\Ab)\psi|^2-\kappa^2|\psi|^2+\frac{\kappa^2}2|\psi|^4+(\kappa
H)^2|\curl\Ab-1)|^2\right)\,\md x\,.
\end{align}
The modulus of the wave  function $\psi$ measures the density of the
superconducting electrons; the curl of the vector field $\Ab$
measures the induced magnetic field; the parameter $H$ measures the
intensity of the external magnetic field and the paramter $\kappa$
is a characteristic of the superconducting material. The functional
in \eqref{eq-3D-GLf} is invariant under gauge transformations, i.e.
if $\chi\in H^1(\Omega;\R)$, then
$$\mathcal E\left(e^{i\chi}\psi,\Ab+\nabla\chi\right)=\mathcal
E(\psi,\Ab)\,.$$
The ground state energy of the functional in \eqref{eq-3D-GLf} is,
\begin{equation}\label{eq-gse}
\E=\inf\{\mathcal E(\psi,\Ab)~:~(\psi,\Ab)\in H^1(\Omega;\mathbb
C)\times H^1(\Omega;\mathbb R^2)\}\,.
\end{equation}
The behavior of the ground state energy $\E$ and of the minimizers
depend strongly on the intensity of the external field \cite{FH-b,
SaSe}. Loosely speaking, there exist three critical values
$H_{C_1}(\kappa)$, $H_{C_2}(\kappa)$ and $H_{C_3}(\kappa)$ such
that, when the parameter $\kappa$ is sufficiently large and
$(\psi_{\kappa,H},\Ab_{\kappa,H})$ is a minimizer of the functional
in \eqref{eq-3D-GLf}, the following is true:
\begin{itemize}
\item If the parameter $H$ satisfies $H<H_{C_1}$, then
$|\psi_{\kappa,H}|>0$ everywhere;
\item if $H_{C_1}<H<H_{C_2}$, then $|\psi_{\kappa,H}|$ has isolated
zeros, called {\it vortices}; these zeros become evenly
distributed in the domain $\Omega$ when $H\gg H_{C_1}$;
\item if $H_{C_2}<H<H_{C_3}$, $|\psi_{\kappa,H}|$ is localized near
the boundary of the domain $\Omega$ (this is the {\it surface}
superconductivity regime);
\item if $H>H_{C_3}$, $|\psi_{\kappa,H}|=0$ everywhere.
\end{itemize}

The two monographs \cite{FH-b} and \cite{SaSe} are completely
devoted to the detailed analysis of the critical fields in the large
$\kappa$ regime, i.e.  $\kappa\to\infty$.  The regime  where
vortices exist (i.e. $H<H_{C_2}$) is analyzed in \cite{SaSe}. The
regime of surface superconductivity above $H_{C_2}$ is the subject
of \cite{FH-b}.

A useful way to distinguish between the various critical fields is
the analysis of the distribution of the energy density in the domain
($e_{\kappa,H}(\psi,\Ab)$ in \eqref{eq-3D-GLf}). This is used in
\cite{Pa02} to distinguish the surface behavior above $H_{C_2}$ and
in \cite{SS02} to study the bulk behavior below $H_{C_2}$.

As a consequence of  the results in \cite{Pa02}, a minimizing order
parameter $\psi$ is localized near the boundary of the domain and
exhibits a boundary layer with a length scale of order
$\kappa^{-1}$. This behavior is valid when $H_{C_2}<H<H_{C_3}$. The
result of \cite{Pa02} is sharpened in  \cite{AlH} and \cite{CN}.

The results of \cite{SS02} are valid when the magnetic field is
comparable with the critical field $H_{C_2}$ and $H<H_{C_2}$. It is
obtained that the energy density is uniformly distributed in the
bulk of the domain thereby suggesting periodicity of  minimizing
order parameters.

In this paper, we investigate the behavior of the minimizers when
$H$ is close to and {\it below} the critical value $H_{C_2}$.
Existing mathematical results \cite{AS, Al, FK-am, Pa02, SS02}
suggest that
$$H_{C_2}(\kappa)=\kappa+o(\kappa)\quad{\rm as~}\kappa\to\infty\,.$$
In \cite{FK-am}, when the parameter $H$ satisfies
$$H=\kappa+o(\kappa)\,,\quad(\kappa\to\infty)\,,$$
it is obtained the following formula for the ground state energy,
\begin{equation}\label{eq-FKam}
\E=E_{\rm surf}|\partial\Omega|\kappa +E_{\rm
Ab}|\Omega|\,[\kappa-H]^2_++o\Big(\max\big(\kappa,[\kappa-H]^2_+\big)\Big)\,,\quad(\kappa\to\infty)\,.\end{equation}
Here $E_{\rm surf}<0$ and $E_{\rm Ab}<0$ are two universal
constants, of which $E_{\rm Ab}$ is related to the celebrated
Abrikosov energy \cite{Ab}; $|\partial\Omega|$ is the arc-length
measure of the boundary and $|\Omega|$ is the Lebesgue (area)
measure of $\Omega$.

The asymptotics in \eqref{eq-FKam} displays the transition from {\it
bulk} to {\it surface} concentration of the energy close to the
critical field $H_{C_2}$. It says little about the concentration of
minimizing order parameters. If $(\psi,\Ab)$ is a minimizer of the
functional in \eqref{eq-3D-GLf}, then it follows from
\eqref{eq-FKam},
\begin{equation}\label{eq-FKam'}
\int_\Omega|\psi|^4\,dx=-\frac{2E_{\rm
surf}}{\kappa}|\partial\Omega|-2E_{\rm
Ab}\,\left[1-\frac{H}{\kappa}\right]_+^2|\Omega|+
o\left(\max\Big(\frac1\kappa,\left[1-\frac{H}{\kappa}\right]^2_+\Big)\right)\,,\quad(\kappa\to\infty)\,.\end{equation}
Clearly, if the parameter $H$ satisfies\,\footnote{The notation
$a(\kappa)\ll b(\kappa)$ means that $a$ and $b$ are positive
functions and $a/b\to 0$ as $\kappa\to\infty$}
$\kappa^{-1/2}\ll1-\frac{H}{\kappa}\ll 1$, then the bulk term in
\eqref{eq-FKam'} is the dominant term. In this case,
\eqref{eq-FKam'} is compatible with the following $L^\infty$-bound
obtained in \cite{FK-am},
\begin{equation}\label{eq-FKam''}
\|\psi\|_{L^\infty(\Omega_{\kappa,\rho})}\leq C
\left[1-\frac{H}{\kappa}\right]^{1/2}\,,\end{equation} where
$\Omega_{\kappa,\rho}=\{x\in \Omega~:~{\rm
dist}(x,\partial\Omega)\geq \kappa^{-1+\rho}\}$, $\rho\in(0,1)$ and
$C$ is a constant. In  this paper we establish the additional
asymptotics of $|\psi|^2$,
\begin{equation}\label{eq-K:main}
\int_\Omega|\psi|^2\,dx= -2{E_{\rm
Ab}}\,\left[1-\frac{H}{\kappa}\right]|\Omega|+
o\left(\left[1-\frac{H}{\kappa}\right]\right)\,,\quad(\kappa\to\infty)\,.\end{equation}
The asymptotics in \eqref{eq-K:main} seems   more relevant to
physicists than the one in \eqref{eq-FKam'}. The density of
superconducting electrons (Cooper pairs) is proportional to
$|\psi|^2$. Consequently, \eqref{eq-K:main} tells us what  the {\it
average} of the density of Cooper pairs in $\Omega$ is. Furthermore,
the right side of \eqref{eq-K:main} displays the intensity of bulk
superconductivity, and  describes how fast superconductivity is
restored in the sample when the magnetic field is gradually
decreased.

The precise statement of the main result of this paper is:
\begin{thm}\label{thm:main}
Suppose that $H$ is a function of $\kappa$, $H< \kappa$  and
$$\kappa^{-1/2}\ll1-\frac{H}{\kappa}\ll 1\,,\quad(\kappa\to\infty)\,.$$
Let $R_0>0$, $R_1>0$, $\rho\in(0,1)$  and $Q_\ell$ be a square of
side length $\ell$ such that,
$$Q_\ell\subset\{x\in \Omega~:~{\rm dist}(x,\partial\Omega)\geq
\kappa^{-1+\rho}\}\,,$$ $\ell$  a function of $\kappa$,
$\displaystyle\frac{\ell^2\kappa H}{2\pi}\in\mathbb N$ and
$R_0\kappa^{-1/2}\leq\ell\leq R_1\kappa^{-1/4}$ for all $\kappa$.

 If
$(\psi,\Ab)$ is a minimizer of the functional in \eqref{eq-3D-GLf},
then,
\begin{equation}\label{eq-psi-K}\frac1{|Q_\ell|}\int_{Q_\ell}|\psi|^2\,dx=-2E_{\rm
Ab}\left[1-\frac{H}{\kappa}\right]+o\left(\left[1-\frac{H}{\kappa}\right]\right)\,,\quad(\kappa\to\infty)\,.\end{equation}
Here $E_{\rm Ab}\in\,[-\frac12,0[$ is a universal constant defined
in \eqref{eq-hc2-E2}.
\end{thm}

A key step to prove  Theorem~\ref{thm:main} is the approximation of
the order parameter $\psi$ by a periodic eigenfunction of the Landau
Hamiltonian (see Theorem~\ref{thm:Al}). In \cite{Al}, such an
approximation is given when $\kappa$ and $H$ satisfy,
$$\kappa^{-2/5}\ll1-\frac{H}{\kappa}\ll\frac1{(\ln\kappa)^2}\,,\quad(\kappa\to\infty)\,.$$
This assumption is restrictive compared to that of
Theorem~\ref{thm:main}. Furthermore, the result of
Theorem~\ref{thm:main} goes beyond the result of \cite{Al} as the
formula \eqref{eq-psi-K} is new.

Thanks to the sharp $L^\infty$ bound in \eqref{eq-FKam''}, we have
the following slight improvements of Theorem~\ref{thm:main}. If the
side length of the square $Q_\ell$ satisfies the relaxed condition
$R_0\kappa^{-1/2}\leq \ell\ll1$, then it can be approximated by
squares of side length satisfying the condition of
Theorem~\ref{thm:main} and the asymptotics in \eqref{eq-psi-K}
remains true. The same remark applies if the square $Q_\ell$ is
replaced by a domain that can be approximated by squares whose side
lengths satisfy the condition in Theorem~\ref{thm:main}. In
particular, if squares in Theorem~\ref{thm:main} are replaced by
disks of radii $\ell$ or parallelograms of side lengths comparable
with $\ell$, then the asymptotics in \eqref{eq-psi-K} remains true.

The condition made on the side-length $\ell$, namely $\ell\geq
R_0\kappa^{-1/2}$, is technical. It is needed to get that  the
remainder terms of the estimates in Propositions~\ref{prop-ub} and
\ref{ub:A-A0} are of lower order compared to the expected leading
order term. The method we use to approximate the energy implies that
the asymptotics in Theorem~\ref{thm:main} is true if and only if one
can select $\delta\in(0,1)$ and $\ell\in(0,1)$  such that,
\begin{equation}\label{eq:disc-cond}
\delta^{-1}\ell^2\kappa+\delta\kappa +\ell^{-1}\ll
[\kappa-H]\,.\end{equation}
Clearly, we observe that the condition $\ell\geq R_0\kappa^{-1/2}$
may be relaxed down to $\ell\gg\kappa^{-1}$ if we know that
$1-\frac{H}{\kappa}$ is not close to $\kappa^{-1/2}$. More
precisely, if  $1\ll h(\kappa)\ll \kappa^{1/2}$,
$\delta=\left(h(\kappa)\left(1-\frac{H}{\kappa}\right)\kappa^{-1}\right)^{1/2}$
 and
$$\frac1{h(\kappa)}\ll 1-\frac{H}{\kappa}\ll
1\,,$$ then  the asymptotics in Theorem~\ref{thm:main} remains true for
$\ell=h(\kappa)\kappa^{-1}$.

There might be a physically relevant reason behinds the technical
point that forces $\ell$ to increase up to the order of
$\kappa^{-1/2}$ when $1-\frac{H}{\kappa}$ approaches
$\kappa^{-1/2}$. In \cite{SS02}, when $H$ is below but not
asymptotically close to $H_{C_2}$, it is constructed test
configurations that hint at the expected behavior of minimizers. As
a consequence, it is expected that the minimizing order parameter
will have {\it vortices} and the core size of each vortex is
proportional to $\kappa^{-1}$. As $H$ approaches $H_{C_2}\sim
\kappa$, the core size of the vortex might increase up to
$\kappa^{-1/2}$. When $H$ is increased further up to
$\kappa-\mu\sqrt{\kappa}$ and $\mu\ll 1$, it is expected that all
vortices will merge into a giant vortex and superconductivity
becomes a surface phenomenon, as is revealed from the energy
asymptotics in \eqref{eq-FKam}. However, the rigorous verification
of the aforementioned picture is open.

It is not likely  that the result of Theorem~\ref{thm:main} extends
to squares $Q_\ell$ that live at a distance of order $\kappa^{-1}$
away from the boundary $\partial\Omega$. It is pointed in
\cite{Pa02} that minimizing order parameters  will be of order $1$
in  a boundary layer of  length scale $\kappa^{-1}$.

There is an interesting consequence of Theorem~\ref{thm:main}.   If
we know that $E_{\rm Ab}=-\frac12$,  then \eqref{eq-FKam'} and
\eqref{eq-K:main} together yield,
\begin{equation}\label{eq-u}
\lim_{\kappa\to\infty}\int_{\Omega}\big(A-|u|^2\big)^2\,dx=0\,,\quad
A=-2E_{\rm Ab}\,,\quad
u=\left[1-\frac{H}{\kappa}\right]^{-1/2}\psi\,.\end{equation} Such a
bound is helpful to construct a vortex structure of $u$ \cite{SaSe}.
However, if $-\frac12<E_{\rm Ab}<0$, a convergence such as the one
in \eqref{eq-u} does not hold and the profile of $|\psi|^2$ is not
homogeneous a.e. in $\Omega$. Existing estimates suggest that
$E_{\rm Ab}<-\frac12$ (see \cite{ABN, Sig}) which rule out the
complete homogeneity of $|\psi|^2$. Notice that this is in agreement
with the expected behavior that $\psi$ will have isolated zeros
arranged in a (triangular) lattice, \cite{Ab}.

We conclude  by clarifying some notation that will be used
throughout this paper. If $a(\kappa)$ and $b(\kappa)$ are two
positive functions, we write $a(\kappa)\approx b(\kappa)$ to mean
that there exist positive constants $\kappa_0$, $c_1$ and $c_2$ such
that $c_1b(\kappa)\leq a(\kappa)\leq c_2b(\kappa)$ for all
$\kappa\geq\kappa_0$. The notation $a(\kappa)\sim b(\kappa)$ means
that $a(\kappa)=b(\kappa)\big(1+o(1)\big)$ as $\kappa \to\infty$.
The notation $a(\kappa)\ll b(\kappa)$ means
$a(\kappa)=o(1)\,b(\kappa)$ as $\kappa\to\infty$. Constants in the
remainder of inequalities are all denoted by the letter $C$, whose
value might change from a line to another.

Finally, notice that in the parameter regime of
Theorem~\ref{thm:main},
$$\left[1-\frac{H}{\kappa}\right]=\left[\frac{\kappa}{H}-1\right]\big(1+o(1)\big)\,,\quad(\kappa\to\infty)\,.$$
This remark will be often used throughout the paper.

\section{Useful estimates}\label{hc2-sec-preliminaries}

In this section, we collect {\it a priori} estimates satisfied by
critical points of the functional in \eqref{eq-3D-GLf}. Notice that
critical points of the functional in \eqref{eq-3D-GLf} satisfy the
Ginzburg-Landau equations:
\begin{equation}\label{eq-hc2-GLeq}
\left\{
\begin{array}{l}
-(\nabla-i\kappa H\Ab)^2\psi=\kappa^2(1-|\psi|^2)\psi\,,\\
-\nabla^\bot\curl\Ab=(\kappa H)^{-1}{\rm
Im}(\overline\psi\,(\nabla-i\kappa
H\Ab)\psi)\,,\quad{\rm in}~\Omega\,,\\
\nu\cdot(\nabla-i\kappa H\Ab)\psi=0\,,\quad \curl\Ab=1\,,\quad{\rm
  on}~\partial\Omega\,.
\end{array}\right.
\end{equation}
Here $\nu$ is the unit inward normal vector of $\partial\Omega$. The
set of estimates in Lemma~\ref{lem-hc2-FoHe} appeared first in
\cite{LuPa99} (for a more particular regime) and were then  proved
for a wider regime in \cite{FoHe06}.

\begin{lem}\label{lem-hc2-FoHe}
There exist positive constants $\kappa_0$ and $C$ such that, if
$\kappa\geq\kappa_0$, $H\geq\frac\kappa2$ and  $(\psi,\Ab)$ is a
solution of (\ref{eq-hc2-GLeq}), then,
\begin{align}\label{eq-hc2-FoHe1}
&\|\curl\Ab-1\|_{C^1(\Omega)}+\kappa^{-1}\|\curl\Ab-1\|_{C^2(\Omega)}\leq
C\kappa^{-1}\,,\\
&\|(\nabla-i\kappa H\Ab)\psi\|_{L^\infty(\Omega)}\leq
C\kappa\,.\label{eq-hc2-FoHe2}
\end{align}
\end{lem}

The sharp $L^\infty$ bound in the next  theorem is established in
\cite{FK-am}.  It has been conjectured in a weaker form in
\cite{AS}. The bound of Theorem~\ref{thm:psi;1-b} plays a key-role
in the proof of Theorem~\ref{thm:main}.

\begin{thm}\label{thm:psi;1-b}
\label{thm:Linfty-bulk} Let $\rho\in (0,1)$. Suppose that $\kappa$
and $H$ satisfy,
$$\kappa^{-1/2}\ll1-\frac{H}{\kappa}\ll
1\,,\quad(\kappa\to\infty)\,.$$ There exist  positive constants $C>0$ and
$\kappa_0$ such that, if $\kappa\geq\kappa_0$ and $(\psi, {\bf A})$
is a solution of \eqref{eq-hc2-GLeq}, then,
\begin{align}\label{eq-psi<1-b}
\| \psi \|_{L^{\infty}(\Omega_{\kappa,\rho})} \leq C \left[1-\frac{H}{\kappa}\right]^{1/2}\,,
\end{align}
where
\begin{align}
\Omega_{\kappa,\rho} := \{ x \in \Omega\,|\, \dist(x,\partial \Omega)
\geq \kappa^{-1+\rho}\}\,.
\end{align}
\end{thm}

\section{The limiting problem}\label{hc2-sec-lbp}

\subsection{Reduced Ginzburg-Landau functional}

Given a constant $b\geq 0$ and an open set $\mathcal D\subset \R^2$,
we define the following Ginzburg-Landau energy,
\begin{equation}\label{eq-LF-2D}
G_{b,\mathcal D}(u)=\int_{\mathcal D}\left(b|(\nabla-i\Ab_0)u|^2
-|u|^2+\frac1{2}|u|^4\right)\,dx\,.
\end{equation}
Here $\Ab_0$ is the canonical magnetic potential with unit constant
magnetic field,
\begin{equation}\label{eq-hc2-mpA0}
\Ab_0(x_1,x_2)=\frac12(-x_2,x_1)\,,\quad\Big(x=(x_1,x_2)\in
\R^2\Big)\,.\end{equation}
We will consider the functional $G_{b,\mathcal D}$ first with
Dirichlet and later with (magnetic) periodic boundary conditions. It
will be clear from the context what is meant.

Consider the functional with Dirichlet boundary conditions and for
$b>0$. If the domain $\mathcal D$ is bounded, completing  the square
in the expression of $G_{b,\mathcal D}$ shows that $G_{b,\mathcal
D}$ is bounded from below. Thus, starting from a minimizing
sequence, it is easy to check that $G_{b,\mathcal D}$ has a
minimizer. A standard application of the maximum principle  shows
that, if $u$ is any minimizer of $G_{b,\mathcal D}$, then
\begin{align}
\label{eq:MaxPr}
|u|\leq 1,\qquad \text{  in } {\mathcal D},
\end{align}
see e.g. \cite{SaSe}.

Given $R>0$, we denote by  $K_R=(-R/2,R/2)\times(-R/2,R/2)$ a square
of side length $R$. Let,
\begin{equation}\label{eq-m0(b,R)}
m_0(b,R)=\inf_{u\in H^1_0(K_R;\C)}
G_{b,K_R}(u)\,.
\end{equation}

The following remark will be useful. If $u_R$ is a minimizer of
\eqref{eq-m0(b,R)}, then $u_R$ satisfies the Ginzburg-Landau
equation,
$$-(\nabla-i\Ab_0)^2u_R=b^{-1}(1-|u_R|^2)u_R\,,\quad {\rm in ~}Q_R\,.$$
Recall that $u_R\in H^1_0(Q_R)$. We  extend $u_R$ by magnetic
periodicity to all $\R^2$, i.e. to a function in the space $E_R$ in
\eqref{eq-hc2-space1} below. That way, $u_R$  satisfies the equation
$$-(\nabla-i\Ab_0)^2u_R=b^{-1}(1-|u_R|^2)u_R\,,$$
in all $\R^2$. We can apply  Theorem~3.1 in \cite{FH-jems} to get
that, \begin{equation}\label{eq:u<1-b} \forall ~b\in
[1-b_0,1[\,,\quad |u_R|\leq C_{\max}\sqrt{1-b}\,,\end{equation}
where $b_0$ and $C_{\max}$ are universal constant.

\subsection{Periodic minimizers}\label{sec:periodic}

We introduce the following space,
\begin{multline}\label{eq-hc2-space1}
E_{R}=\bigg{\{}u\in H^1_{\rm loc}(\R^2;\C)~:~u(x_1+R,x_2)=e^{i
Rx_2/2 }u(x_1,x_2),\\
u(x_1,x_2+R )=e^{-iRx_1/2
}u(x_1,x_2)\,,~\Big((x_1,x_2)\in\R^2\Big)\bigg{\}}\,.
\end{multline}

Notice that the periodicity conditions in \eqref{eq-hc2-space1} are
constructed in such a manner that all physically relevant quantities
are periodic (i.e. density, energy and super-current). More
precisely, for any function $u\in E_{R}$, the functions $|u|$,
$|\nabla_{\Ab_0}u|$ and the vector field $\overline u
\nabla_{\Ab_0}u$ are periodic with respect to the lattice generated
by $K_R$.

Recall the functional $G_{b,\mathcal D}$ in \eqref{eq-LF-2D} above.
We introduce the ground state energy,
\begin{equation}\label{eq-mp(b,R)}
m_{\rm p}(b,R)=\inf_{u\in E_R}
G_{b,K_R}(u)\,.
\end{equation}

The next proposition exhibits a relation between the ground state
energies  $m_0(b,R)$ and $m_{\rm p}(b,R)$, namely that $m_{\rm
p}(b,R)$ is a valid approximation of $m_0(b,R)$ when $[1-b]R\ll1$.
It is proved in \cite{FK-cpde}.

\begin{prop}\label{prop-m0=mp}
Let  $m_0(b,R)$ and $m_{\rm p}(b,R)$ be as introduced in
\eqref{eq-m0(b,R)} and \eqref{eq-mp(b,R)} respectively. For all
$b>0$ and $R>0$,  we have,
$$m_0(b,R)\geq m_{\rm p}(b,R)\,.$$
Furthermore, there exist universal constants $\epsilon_0\in(0,1)$
and $C>0$ such that, if $b\geq 1-\epsilon_0$ and $R\geq 2$, then,
\begin{align}\label{eq:m0mP}
m_0(b,R)\leq m_{\rm p}(b,R)+C[1-b]_+R\,.
\end{align}
\end{prop}

\subsection{The periodic Schr\"odinger operator with constant
  magnetic field.}

In this section, we assume the quantization condition that
$|K_{R}|/(2\pi)$ is an integer, i.e. there exists $N\in\mathbb N$
such that,
\begin{equation}\label{eq-hc2-quantization}
R^2=2\pi N\,.\end{equation} Recall the  magnetic potential $\Ab_0$
introduced in \eqref{eq-hc2-mpA0} above. Consider the operator,
\begin{equation}\label{eq-hc2-poperator}
P_{R}=-(\nabla-i\Ab_0)^2\quad{\rm in}~L^2(K_{R})\,,
\end{equation}
with form domain the space $E_{R}$ introduced in
(\ref{eq-hc2-space1}). More precisely, $P_{R}$ is the self-adjoint
realization associated with the closed quadratic form
\begin{equation}\label{eq-hc2-poperatorQF}
E_{R}\ni f\mapsto Q_{R}(f)=\|(\nabla-i\Ab_0)f\|_{L^2(K_{R})}^2\,.
\end{equation}

The operator $P_{R}$ has compact resolvent. Denote by
$\{\mu_j(P_{R})\}_{j\geq1}$ the increasing sequence of its distinct
eigenvalues (i.e. without counting multiplicity).

The following proposition may be classical in the spectral theory
of Schr\"odinger operators, but we refer to \cite{AS} or \cite{Al}
for a simple proof.

\begin{prop}\label{prop-hc2-poperator}
Assuming $R$  is such that $|K_{R}|\in2\pi\mathbb N$,  the operator
$P_{R}$ enjoys the following spectral properties:
\begin{enumerate}
\item $\mu_1(P_{R})=1$ and $\mu_2(P_{R})\geq 3$\,.
\item The space $L_{R}={\rm Ker}(P_{R}-1)$ is finite
dimensional and ${\rm dim}\,L_{R}=|K_{R}|/(2\pi)$\,.
\end{enumerate}
Consequently, denoting by $\Pi_1$ the orthogonal projection on the
space $L_{R}$ (in $L^2(K_{R})$) and by $\Pi_2={\rm Id}-\Pi_1$, we
have for all $f\in D(P_{R})$,
$$\langle P_{R}\Pi_2 f\,,\,\Pi_2f\rangle_{L^2(K_{R})}\geq
3\|\Pi_2f\|^2_{L^2(K_{R})}\,.$$
\end{prop}

The next lemma is a consequence of the existence of a spectral gap
between the first two eigenvalues of $P_{R}$. It is proved in
\cite{FK-am}.

\begin{lem}\label{prop-hc2-poperator'}
Given $p\geq 2$, there exists a constant $C_p>0$ such that, for any
$\gamma\in(0,\frac12)$, $R\geq 1$ and $f\in D(P_{R})$ satisfying
\begin{equation}\label{eq-hc2-hypf}
Q_{R}(f)-(1+\gamma)\|f\|^2_{L^2(K_{R})}\leq0\,,\end{equation} the
following estimate holds,
\begin{equation}\label{eq-hc2-1=proj}
\|f-\Pi_1f\|_{L^p(K_{R})}\leq
C_p\sqrt{\gamma}\,\|f\|_{L^2(K_{R})}\,.
\end{equation}
Here $\Pi_1$ is the projection on the space $L_{R}$.
\end{lem}

\subsection{The Abrikosov energy.} We introduce the
following  energy functional (the Abrikosov energy),
\begin{equation}\label{eq-hc2-eneAb}
F_{R}(v)=\int_{K_{R}}\left(\frac12|v|^4-|v|^2\right)\,d
x\,.
\end{equation}
The energy $F_R$ will be minimized on the space
$L_{R}$, the eigenspace of the first eigenvalue of the periodic
operator $P_{R}$,
\begin{align}\label{eq-hc2-space2}
L_{R}&=\{u\in E_{R}~:~P_{R}u=u\}\,.
\end{align}

We need the following theorem which we take from \cite{AS, FK-cpde}.

\begin{thm}\label{thm-AS}
Let
\begin{equation}\label{eq-hc2-c(r,t)}
\forall~R>0\,,\quad c(R)=\min\{F_{R}(u)~:~u\in L_{R}\}\,.
\end{equation}
There exists a constant $E_{\rm Ab}\in[-\frac12,0[$ such that,
\begin{equation}\label{eq-hc2-E2}
E_{\rm Ab}=\lim_{\substack{R\to\infty\\R^2/(2\pi)\in\mathbb
N}}\frac{c(R)}{R^2}\,.\end{equation}
\end{thm}

The energy $c(R)$ is a specific Abrikosov energy corresponding to
the square lattice. The Abrikosov energy can be defined over any
parallelogram lattice and is minimized for the triangular lattice,
\cite{AS, Sig}. In the regime of large area $R\to\infty$, the
lattice shape is unimportant to leading order, \cite{AS}.

It is observed in \cite{FK-cpde} that there is a relationship
between the ground state energies $m_{\rm p}(b,R)$ and $c(R)$,
namely that $[1-b]^2c(R)$ is a valid approximation of $m_{\rm
p}(b,R)$ in the regime $[1-b]^2R^4\ll 1$. This is recalled in the
next theorem.

\begin{thm}\label{thm-mp=c}
Let $m_{\rm p}(b,R)$ and $c(R)$ be as   introduced in
\eqref{eq-mp(b,R)} and \eqref{eq-hc2-c(r,t)} respectively. For all
$b>0$ and $R>0$, we have,
$$m_{\rm p}(b,R)\leq [1-b]_+^2c(R)\,.$$
Furthermore, there exist universal constants $\epsilon_0\in(0,1)$
and $C>0$ such that, if $R\geq 2$, $b\geq 1-\epsilon_0$,  and
$0<\sigma<1/2$, then,
$$m_{\rm p}(b,R)\geq [1-b]_+^2\bigg((1+2\sigma)c(R)-C\sigma^{-3}(1-b)^2R^4\bigg)\,.$$
\end{thm}

\section{Energy in small squares}\label{hc2-sec-ub}

In this section, the notation  $Q_\ell$ stands for  a square in
$\R^2$ of side length $\ell>0$
$$Q_\ell=(-\ell/2+a_1,a_1+\ell/2)\times (-\ell/2+a_2,a_2+\ell/2)\,,$$
where $a=(a_1,a_2)\in\R^2$.

If $(\psi,\Ab)\in H^1(\Omega;\C)\times  H^1(\Omega;\R^2)$, we denote
by $e(\psi,\Ab)=|(\nabla -i\kappa
H\Ab)\psi|^2-\kappa^2|\psi|^2+\frac{\kappa^2}2|\psi|^4$.
Furthermore, we define the Ginzburg-Landau energy of $(\psi,\Ab)$ in
a domain $\mathcal D\subset\Omega$ as follows,
\begin{equation}\label{eq-GLen-D}
\mathcal E(\psi,\Ab;\mathcal D)=\int_{\mathcal D} e(\psi,\Ab)\,dx+(\kappa H)^2\int_{\Omega}|\curl\Ab-1|^2\,dx\,.
\end{equation}
Also we introduce the functional,
\begin{equation}\label{eq-GLe0}
\mathcal E_0(\psi,\Ab;\mathcal D)=\int_{\mathcal D}\left(|(\nabla-i\kappa H\Ab)\psi|^2-\kappa^2|\psi|^2+\frac{\kappa^2}2|\psi|^4\right)\,dx\,.
\end{equation}
The results of this section will be derived under the assumption
that the magnetic field $H$ satisfies,
\begin{equation}\label{eq-as-H1}
H=\kappa-\mu(\kappa)\sqrt{\kappa}\,,\end{equation} where the
function $\mu(\kappa)$ satisfies,
\begin{equation}\label{eq-as-H2}
\liminf_{\kappa\to\infty}\mu(\kappa)=\infty\quad{\rm and}\quad
\limsup_{\kappa\to\infty}\frac{\mu(\kappa)}{\sqrt{\kappa}}=0\,.\end{equation}
The assumptions \eqref{eq-as-H1}-\eqref{eq-as-H2} are equivalent to
those in Theorem~\ref{thm:main}. As mentioned in the introduction,
\eqref{eq-as-H1}-\eqref{eq-as-H2} cover a range of the parameter $H$
wider than the one assumed in \cite{Al}. In that direction, the
results here are stronger than those of \cite{Al}.

\begin{prop}\label{prop-ub}
Suppose that the magnetic field $H$ satisfies \eqref{eq-as-H1} and
\eqref{eq-as-H2}. There exist positive constants $C$, $R_0$ and
$\kappa_0$ such that the following is true. Let $\kappa$, $\ell$,
$\rho$
 and $\delta$ satisfy $\kappa\geq \kappa_0$, $R_0 \kappa^{-1} \leq \ell
\leq 1/2$, $\rho\in(0,1)$ and $\delta\in\,(0,1)$.  If $(\psi,\Ab)\in
H^1(\Omega;\C)\times H^1(\Omega;\R^2)$ is a minimizer of
\eqref{eq-3D-GLf}, and $Q_{\ell}\subset\Omega$ is a square of side
length $\ell$ satisfying,
$$ Q_\ell\subset \Omega_{\kappa,\rho}=\{x\in\Omega~:~{\rm
dist}(x,\partial\Omega)\geq \kappa^{-1+\rho}\}\,,$$
 then,
$$
\frac1{|Q_\ell|}\mathcal E_0(\psi,\Ab;Q_{\ell})\leq (1+\delta)\,
[\kappa-H]^2\,\frac{c(R)}{R^2}
+C\Big(\delta^{-1}\ell^2\kappa+\delta\kappa+\ell^{-1}\Big) [\kappa
-H]\,.$$ Here $R=\sqrt{\kappa H}\ell$,  $c(R)$ is the
function introduced in \eqref{eq-hc2-c(r,t)} and $\mathcal E_0$ is
the functional introduced  in \eqref{eq-GLe0}.
\end{prop}
\begin{proof}
Notice that the energy $\mathcal E_0$ is invariant under gauge
transformations\break $(\psi,\Ab)\mapsto
(e^{i\chi}\psi,\Ab+\nabla\chi)$. After performing a gauge
transformation, we may suppose that the magnetic potential $\Ab$
satisfies (see \cite[(5.31)]{FK-am}),
\begin{equation}\label{eq-3D-est-g'}
|\Ab(x)-\Ab_0(x)|\leq
\frac{C\ell}{\sqrt{\kappa H}}\,,\quad(x\in Q_{\ell})\,,
\end{equation}
where $\Ab_0$ is the magnetic potential introduced in
\eqref{eq-hc2-mpA0}.

Without loss of generality, we may assume that,
$$Q_{\ell}=(-\ell/2,\ell/2)\times (-\ell/2,\ell/2)\subset\Omega_{\kappa,\rho}\,.$$
Let  $b=H/\kappa$, $R=\ell\sqrt{\kappa H}$ and $u_R\in H^1_0(Q_R)$
be a minimizer of the functional $\mathcal G_{b,Q_R}$ introduced in
\eqref{eq-LF-2D}, i.e. $\mathcal G_{b,Q_R}(u_R)=m_0(b,R)$ where
$m_0(b,R)$ is introduced in \eqref{eq-m0(b,R)}.

Let $\chi_R\in C_c^\infty(\R^2)$ be a cut-off function  such that,
$$0\leq\chi_R\leq 1\quad{\rm in~}\R^2\,,\quad {\rm supp}\,\chi_R\subset Q_{R+1}\,,\quad  \chi_R=1\quad{\rm in~}
Q_{R}\,.$$ and  $|\nabla\chi_R|\leq C$ for some universal constant
$C$. Let $\eta_R(x)=1-\chi_R(x\sqrt{\kappa H})$ for all $x\in\R^2$.
Recall that $(\psi,\Ab)$ is a minimizer of the functional in
\eqref{eq-3D-GLf}. We introduce the function (whose construction is
inspired from \cite{SS02}),
\begin{equation}\label{eq-test-conf}
\varphi(x)=\mathbf 1_{Q_{\ell}}(x) u_R(x\sqrt{\kappa H})+\eta_R(x) \psi(x)\,,\quad (x\in\Omega)\,.
\end{equation}
Notice that by construction, $\varphi$ satisfies,
$$\varphi(x)=\left\{
\begin{array}{ll}
u_R(\,x\sqrt{\kappa H}\,)&{\rm if~}x\in Q_{\ell}\,,\\
\eta_R(\,x\sqrt{\kappa H}\,)\psi(x)&{\rm if~}x\in Q_{\ell+\frac1{\sqrt{\kappa H}}}\setminus Q_{\ell}\,,\\
\psi(x)&{\rm if~}x\in \Omega\setminus Q_{\ell+\frac1{\sqrt{\kappa
H}}}\,.\end{array}\right.$$ This allows us to get that, for all
$\delta\in(0,1)$ (see \cite[(4.13)]{FK-cpde},
\begin{equation}\label{eq-GLub}
\mathcal E(\varphi,\Ab;\Omega)\leq \mathcal E(\psi,\Ab;\Omega\setminus Q_{\ell})+\frac{(1+\delta)}{b}
m_0(b,R)+ r_0(\kappa)\,,
\end{equation}
where $m_0(b,R)$ is defined in \eqref{eq-m0(b,R)}, and for some
constant $C$, $r_0(\kappa)$ is given as follows,
\begin{multline}\label{eq-ub-r0}
r_0(\kappa)=C\Big[\delta^{-1}(\kappa
H)^2\|(\Ab-\Ab_0)u_R\|_{L^2(Q_\ell)}^2+\delta\ell^2\kappa^2\|u_R\|_\infty^2\Big]\\
+\left[\mathcal E_0\big(\,\eta_R(x\sqrt{\kappa H}\,)\psi,\Ab;
Q_{\ell+\frac1{\sqrt{\kappa H}}}\setminus Q_\ell\big)-\mathcal
E_0\big(\psi,\Ab;Q_{\ell+\frac1{\sqrt{\kappa H}}}\setminus
Q_\ell\big)\right]\,.
\end{multline}
Since $(\psi,\Ab)$ is a minimizer, we have,
$$\mathcal E(\psi,\Ab)\leq \mathcal E(\varphi,\Ab;\Omega)\,.$$
Since $\mathcal E(\psi,\Ab;\Omega)=\mathcal
E(\psi,\Ab;\Omega\setminus Q_{\ell})+\mathcal
E_0(\psi,\Ab;Q_{\ell})$, the estimate \eqref{eq-GLub} gives us,
$$\mathcal  E_0(\psi,\Ab;Q_{\ell})\leq \frac{(1+\delta)}{b}
m_0(b,R)+ r_0(\kappa)\,.$$ We use the estimates in \eqref{eq:m0mP}
and Theorem~\ref{thm-mp=c} to write,
\begin{equation}\label{eq-ub-final}
\mathcal E_0(\psi,\Ab,Q_{\ell})\leq
\frac{(1+\delta)}{b}\Big([1-b]_+^2\,c(R)
+C[1-b]_+R\Big)+r_0(\kappa)\,.\end{equation} Next we control the
error term $r_0(\kappa)$. The first term in $r_0(\kappa)$ is
controlled by using  \eqref{eq:u<1-b} and \eqref{eq-3D-est-g'}. That
way we write,
$$\delta^{-1}(\kappa
H)^2\|(\Ab-\Ab_0)u_R\|_{L^2(Q_\ell)}^2+\delta\ell^2\kappa^2\|u_R\|_\infty^2\leq
C\Big(\delta^{-1}\ell^4\kappa^2+\delta\ell^2\kappa^2\Big)[1-b]\,.$$ The second
term in $r_0(\kappa)$ is controlled as follows. An integration by
parts allows us to write,
\begin{multline*}
\mathcal E_0\big(\,\eta_R(x\sqrt{\kappa
H}\,)\psi,\Ab;Q_{\ell+\frac1{\sqrt{\kappa H}}}\setminus
Q_\ell\big)=\frac{\kappa^2}{2}\int_{Q_{\ell+\frac1{\sqrt{\kappa
H}}}\setminus Q_\ell}\left(\eta_R^4(x\sqrt{\kappa
H})-2\eta_R^2(x\sqrt{\kappa
H})\right)|\psi|^4\,dx\\+\int_{Q_{\ell+\frac1{\sqrt{\kappa
H}}}\setminus Q_\ell}|\nabla\,\eta_R(x\sqrt{\kappa
H})|^2\,|\psi|^2\,dx\,.
\end{multline*}
Consequently, we get
\begin{align*}
\mathcal E_0\big(\,\eta_R(x\sqrt{\kappa
H}\,)\psi,\Ab;Q_{\ell+\frac1{\sqrt{\kappa H}}}\setminus
Q_\ell\big)&-\mathcal E_0\big(\psi,\Ab;Q_{\ell+\frac1{\sqrt{\kappa H}}}\setminus
Q_\ell\big)\\
&=\frac{\kappa^2}{2}\int_{Q_{\ell+\frac1{\sqrt{\kappa H}}}\setminus
Q_\ell}\left(\eta_R^4(x\sqrt{\kappa H})-2\eta_R^2(x\sqrt{\kappa H})-1\right)|\psi|^4\,dx\\
&\hskip0.5cm+\kappa^2\int_{Q_{\ell+\frac1{\sqrt{\kappa H}}}\setminus
Q_\ell}|\psi|^2\,dx-\int_{Q_{\ell+\frac1{\sqrt{\kappa H}}}\setminus
Q_\ell}|(\nabla-i\kappa H\Ab)\psi|^2\,dx\\
&\hskip0.5cm+\int_{Q_{\ell+\frac1{\sqrt{\kappa H}}}\setminus
Q_\ell}|\nabla\,\eta_R(x\sqrt{\kappa H})|^2\,|\psi|^2\,dx\\
&\leq \kappa^2\int_{Q_{\ell+\frac1{\sqrt{\kappa H}}}\setminus
Q_\ell}|\psi|^2\,dx+C\ell\kappa[1-b]\,.
\end{align*}
We use Theorem~\ref{thm:Linfty-bulk} to write,
$$
\mathcal E_0\big(\,\eta_R(x\sqrt{\kappa
H}\,)\psi,\Ab;Q_{\ell+\frac1{\sqrt{\kappa H}}}\setminus
Q_\ell\big)-\mathcal E_0\big(\psi,\Ab;Q_{\ell+\frac1{\sqrt{\kappa H}}}\setminus
Q_\ell\big)\leq C\ell\kappa\,[1-b]\,.
$$
Therefore, the term $r_0(\kappa)$ satisfies,
$$r_0(\kappa)\leq
C\Big(\delta^{-1}\ell^2\kappa^2+\delta\kappa^2+\ell^{-1}\kappa\Big)
[1-b]\,\ell^2\,.$$ Remembering the definition of $b=H/\kappa$ and the assumption in \eqref{eq-as-H1}-\eqref{eq-as-H2} on $H$, we
finish the proof of Proposition~\ref{prop-ub}. \end{proof}

The proof of the next proposition is similar to that of
Proposition~\ref{prop-ub}.

\begin{prop}\label{ub:A-A0}
Suppose that the assumptions in Proposition~\ref{prop-ub} are true.
Let $\chi_\ell\in C_c^\infty(Q_\ell)$ be a cut-off function
satisfying,
$$\chi_\ell=1\quad{\rm in~}Q_{\ell-\frac1{\sqrt{\kappa H}}}\,,\quad
0\leq \chi_\ell\leq 1\,,\quad |\nabla \chi_\ell|\leq c\sqrt{\kappa
H}\quad{\rm in ~}Q_\ell\,,$$ where $c$ is a universal constant.

There exists a constant $C$ such that, if $(\psi,\Ab)\in
H^1(\Omega;\C)\times H^1(\Omega;\R^2)$ is a minimizer of
\eqref{eq-3D-GLf}, then,
$$
\frac1{|Q_\ell|}\mathcal E_0(\chi_\ell\,\psi,\Ab_0;Q_{\ell})\leq
\frac{1+\delta}{|Q_\ell|}\,\mathcal E_0(\psi,\Ab;Q_{\ell})
+C\Big(\delta^{-1}\ell^2\kappa+\delta\kappa+\ell^{-1}\Big) [\kappa
-H]\,.
$$
\end{prop}

\begin{rem}\label{rem:ub}
Let $R_0>0$, $R_1>0$, and $A$ and $B$  be two functions of $\kappa$
such that,
$$R_0\leq A\leq R_1\kappa^{1/4}\quad {\rm and} \quad 1\leq B\ll \mu\,,$$
where $\mu$ is as in \eqref{eq-as-H1}-\eqref{eq-as-H2}.

 The choice
$\delta=B\kappa^{-1/2}$ and $
\ell=\displaystyle\frac{[\hskip-0.05cm[\, A\kappa^{-1/2}\sqrt{\kappa
H} \,]\hskip-0.05cm]}{\sqrt{2\pi}\,\sqrt{\kappa H}}\approx
A\kappa^{-1/2}$ makes the error terms in Propositions~\ref{prop-ub}
and \ref{ub:A-A0} of order $o\Big([\kappa -H]^2\Big)$. Here
$[\hskip-0.05cm[\,\cdot\,]\hskip-0.05cm]$ is the floor function
(integer part). The choice of $\ell$ forces $R=\ell\sqrt{\kappa H}$
to satisfy $R^2\in 2\pi\mathbb N$. This condition is needed to use
the results of Section~\ref{sec:periodic}.

The above choice  explains the assumption made on $\ell$ in
Theorem~\ref{thm:main}.
\end{rem}

In the sequel, we suppose that the parameters $\delta$ and $\ell$
are selected as in Remark~\ref{rem:ub}.

\begin{thm}\label{thm:loc-en}
Let $\rho\in(0,1)$. Suppose that the magnetic field $H$ satisfies
\eqref{eq-as-H1} and \eqref{eq-as-H2}.  Let $(\psi,\Ab)\in
H^1(\Omega;\C)\times H^1(\Omega;\R^2)$ be a minimizer of
\eqref{eq-3D-GLf} and $Q_{\ell}\subset\Omega$  a square of side
length $\ell$  such that,
$$ Q_\ell\subset \Omega_{\kappa,\rho}=\{x\in\Omega~:~{\rm
dist}(x,\partial\Omega)\geq \kappa^{-1+\rho}\}\,,$$ and $\ell$  given in Remark~\ref{rem:ub}. As
$\kappa\to\infty$, there holds:
\begin{enumerate}
\item
$$\frac1{|Q_\ell|}\mathcal E_0(\psi,\Ab;Q_\ell)=E_{\rm
Ab}[\kappa-H]^2+o\big([\kappa-H]^2\big)\,.$$
\item
$$\frac1{|Q_\ell|}\int_{Q_\ell}|\psi|^4\,dx=-2E_{\rm
Ab}\left[1-\frac{H}{\kappa}\right]^2+o\left(\left[1-\frac{H}{\kappa}\right]^2\right)\,.$$
\item If $\chi_\ell$ is the cut-off function in
Proposition~\ref{ub:A-A0}, then
$$\frac1{|Q_\ell|}\int_{Q_\ell}|\chi_\ell\psi|^4\,dx=-2E_{\rm
Ab}\left[1-\frac{H}{\kappa}\right]^2+o\left(\left[1-\frac{H}{\kappa}\right]^2\right)\,.$$
\end{enumerate}
\end{thm}
\begin{proof}
We collect the estimates of Proposition~\ref{prop-ub} and
\ref{ub:A-A0} together with the discussion in Remark~\ref{rem:ub} to
write,
$$\frac1{|Q_\ell|}\mathcal
E_0(\psi,\Ab;Q_\ell)=[\kappa-H]^2\left(\frac{c(R)}{R^2}\big(1+o(1)\big)\right)+o\big([\kappa-H]^2\big)\,.$$
Theorem~\ref{thm-AS} tells us that
$\displaystyle\frac{c(R)}{R^2}=E_{\rm Ab}+o(1)$. Thus,
$$\frac1{|Q_\ell|}\mathcal
E_0(\psi,\Ab;Q_\ell)=[\kappa-H]^2\,E_{\rm
Ab}+o\big([\kappa-H]^2\big)\,.$$
Next we multiply the first Ginzburg-Landau equation in
\eqref{eq-hc2-GLeq} by $\overline\psi$ and integrate by parts over
the square $Q_\ell$ to get,
$$-\frac{\kappa^2}2\int_{Q_\ell}|\psi|^4\,dx=\mathcal
E_0(\psi,\Ab;Q_\ell) +\int_{\partial Q_\ell}\nu\cdot(\nabla-i\kappa
H\Ab)\psi\,\overline{\psi}\,dx\,.$$ Thanks to the estimates in  \eqref{eq-psi<1-b}, \eqref{eq-hc2-FoHe2} and the choice of $\ell$ in Remark~\ref{rem:ub},
the boundary term is,
$$\int_{\partial Q_\ell}\nu\cdot(\nabla-i\kappa
H\Ab)\psi\,\overline{\psi}\,dx=\mathcal
O\left(\ell\kappa\left[1-\frac{H}{\kappa}\right]^{1/2}\right)=o\left(\ell^2[\kappa-H]^2\right)\,.$$
Thus,
$$-\frac1{|Q_\ell|}\frac{\kappa^2}2\int_{Q_\ell}|\psi|^4\,dx=\frac1{|Q_\ell|}\mathcal
E_0(\psi,\Ab;Q_\ell)+o\big([\kappa-H]^2\big)=[\kappa-H]^2\,E_{\rm
Ab}+o\big([\kappa-H]^2\big)\,.$$ Finally, the assumption
 on the support of the function $\chi_\ell$, the bound
\eqref{eq-psi<1-b} and the choice of $\ell$ together give us,
$$\int_{Q_\ell}|\chi_\ell\psi|^4\,dx=\int_{Q_\ell}|\psi|^4\,dx+\mathcal
O\left(\frac{\ell}{\sqrt{\kappa H}}\left[1-\frac{H}{\kappa}\right]^2\right)=o\left(\ell^2\left[1-\frac{H}{\kappa}\right]^2\right)\,.$$
\end{proof}

The next result is an extension of the result in \cite{Al}. The
improvement is that the result here holds for an extended regime of
$H$.

\begin{thm}\label{thm:Al}
Let $\rho\in(0,1)$. Suppose that the magnetic field $H$ satisfies
\eqref{eq-as-H1} and \eqref{eq-as-H2}. There exist positive
constants $C$ and $\kappa_0$ such that the following is true.

Let $\kappa$ satisfy $\kappa\geq \kappa_0$. Let $(\psi,\Ab)\in
H^1(\Omega;\C)\times H^1(\Omega;\R^2)$ be a minimizer of
\eqref{eq-3D-GLf}, and $Q_{\ell}\subset\Omega$  a square of side
length $\ell$ and center $a_j$ such that,
$$ Q_\ell\subset \Omega_{\kappa,\rho}=\{x\in\Omega~:~{\rm
dist}(x,\partial\Omega)\geq \kappa^{-1+\rho}\}\,.$$ Let $\chi_\ell$ be the function in Proposition~\ref{ub:A-A0}. Define the function
$$v(x)=\left(\chi_\ell\psi\right)\Big(a_j+\frac{x}{\sqrt{\kappa H}}\Big)\,,\quad \big(\,x\in
K_R=\,(-R/2,R/2)\,\times\,(-R/2,R/2)\,\big)\,.$$ There holds,
$$\|v-\Pi_1(v)\|_{L^p(K_R)}\leq
C\sqrt{1-\frac{H}{\kappa}}\,\|v\|_{L^2(K_R)}\,,\quad
(p\in\{2,4\})\,,$$ and
\begin{align*}
\mathcal E_0(\psi,\Ab_0;Q_\ell)&\geq\int_{K_R}\Big(\left(1-\frac{\kappa}{H}\right)|\Pi_1(v)|^2+\frac{\kappa}{2H}|v|^4\Big)\,dx\\
&\geq \left[1-\frac{\kappa}{H}\right]_+^2\bigg(\left(1+2\left(\sigma+1-\frac{\kappa}{H}\right)\right)c(R)-C\sigma^{-3}\left(1-\frac{H}{\kappa}\right)^2R^4\bigg)\,,\quad\big(\sigma\in(0,1/2)\big)\,.
\end{align*}
 Here $\Pi_1$ is the projection introduced in
Proposition~\ref{prop-hc2-poperator}.
\end{thm}
\begin{proof}
Applying a translation, we may suppose that the center of $Q_\ell$
is $a_j=0$ (this amounts to a gauge transformation). We may select
$\kappa$ sufficiently large so that $R=\ell\sqrt{\kappa H}$ lives in
any preassigned neighborhood of infinity. That way, we have
$$\frac{c(R)}{R^2}=E_{\rm Ab}\big(1+o(1)\big)\leq \frac{E_{\rm Ab}}2<0\,.$$
As a consequence, we get from Propositions~\ref{prop-ub} and
\ref{ub:A-A0} that,
$$\int_{Q_\ell}\Big(|\nabla-i\kappa H\mathbf
A_0)\chi_\ell\psi|^2-\kappa^2|\chi_\ell\psi|^2\Big)\,dx<0\,.$$ The change of variable $x\mapsto  \sqrt{\kappa
H}\,(x-a_j)$ yields,
$$\int_{K_R}\Big(|(\nabla-i\mathbf
A_0)v|^2-(1+\gamma)|v|^2\Big)\,dx<0\,,$$ with $\gamma=\frac{\kappa}{H}-1\approx 1-\frac{H}{\kappa}$. The first estimate of Theorem~\ref{thm:Al} follows by applying Lemma~\ref{prop-hc2-poperator'}.

Next we prove the remaining estimates of Theorem~\ref{thm:Al}.
Notice that the change of variable $x\mapsto \sqrt{\kappa
H}\,(x-a_j)$ and Proposition~\ref{prop-hc2-poperator} together tell
us,
\begin{align}
\mathcal E_0(\chi_\ell\psi,\Ab_0;Q_\ell)&=\int_{K_R}\left(|(\nabla-i\Ab_0)v|^2-\frac{\kappa}{H}|v|^2+\frac{\kappa}{2H}|v|^4\right)\,dx\nonumber\\
&\geq \int_{K_R}\Big(\left(1-\frac{\kappa}{H}\right)|\Pi_1(v)|^2+\frac{\kappa}{2H}|v|^4\Big)\,dx\,.\label{proof:FK}
\end{align}
Let $b=\kappa/H$. Recall that $\psi$ satisfies in $Q_\ell$ the
pointwise bound
$$|\psi|\leq C\left[1-\frac{H}{\kappa}\right]^{1/2}\approx
[b-1]^{1/2 }\,.$$ Consequently, $|v|\leq C[b-1]^{1/2}$. This is the
key estimate to finish the proof of Theorem~\ref{thm:Al}. The method
used is the same as that of \cite[Theorem~2.11]{FK-cpde}.

We established that $\|v-\Pi_1(v)\|_{L^4(K_R)}\leq
C\sqrt{1-\displaystyle\frac{\kappa}{H}}\,\|v\|_{L^2(K_R)}$. This
inequality gives us that,
$$\|v\|_{L^4(K_R)}\geq \|\Pi_1(v)\|_{L^4(K_R)}-C\sqrt{b-1}\|v\|_{L^2(K_R)}\,.$$
As a consequence, we get with a new constant $C$ and for all
$\sigma\in(0,1/2)$,
\begin{equation}\label{eq-u>pi1}
\|v\|^4_{L^4(K_R)}\geq (1-\sigma)\|\Pi_1(v)\|^4_{L^4(K_R)}-C\sigma^{-3}(b-1)^2\|v\|^4_{L^2(K_R)}\,.
\end{equation}
Using the pointwise bound of $v$, $|v|\leq C[b-1]^{1/2}$, we get
that,
$$
\|v\|^4_{L^4(K_R)}\geq (1-\sigma)\|\Pi_1(v)\|^4_{L^4(K_R)}-C\sigma^{-3}[1-b]^4R^4\,.
$$
We use this bound to get a lower bound of the term in
\eqref{proof:FK}. That way we get that,
\begin{multline*}
\int_{K_R}\Big(\left(1-\frac{\kappa}{H}\right)|\Pi_1(v)|^2+\frac{\kappa}{2H}|v|^4\Big)\,dx\\
\geq
\int_{K_R}\left(-(b-1)|\Pi_1(v)|^2+\frac12(1-\sigma+\gamma)|\Pi_1(v)|^4\right)\,dx
-C\sigma^{-3}[1-b]^4R^4\,,
\end{multline*}
where $\gamma=\frac{\kappa}{H}-1$. By introducing the new function
$u\in L_R$ as follows,
$$\Pi_1v=\left(\frac{[b-1]}{1-\sigma+\gamma}\right)^{1/2} u\,,$$
we get that
$$\int_{K_R}\left(-(b-1)|\Pi_1(v)|^2+\frac12(1-\sigma+\gamma)|\Pi_1(v)|^4\right)\,dx=\frac{[1-b]^2}{1-\sigma+\gamma}\,F_R(u)\,.$$
Using the lower bound $F_R(u)\geq c(R)$ finishes the proof of
Theorem~\ref{thm:Al}.
\end{proof}

%
%

\begin{proof}[Proof of Theorem~\ref{thm:main}]
Notice that \eqref{eq-as-H1}-\eqref{eq-as-H2} ensure that
$$\left[\frac{\kappa}{H}-1\right]=\left[1-\frac{H}{\kappa}\right]\big(1+o(1)\big)\,,\quad{\rm
as~}\kappa\to\infty\,.$$ Combining the results of Propositions~\ref{prop-ub}-\ref{ub:A-A0} and Theorem~\ref{thm:Al}, we get that,
\begin{align*}
\frac1{|Q_\ell|}\int_{K_R}\Big(\left(1-\frac{\kappa}{H}\right)|\Pi_1(v)|^2+\frac{\kappa}{2H}|v|^4\Big)\,dx&\leq [\kappa-H]^2\Big(
\frac{c(R)}{R^2}+o(1)\Big)\\
&=[\kappa-H]^2\Big(E_{\rm Ab}+o(1)\Big)\,.
\end{align*}
Notice that it is used the asymptotics in Theorem~\ref{thm-AS}.
Using the estimate $\|v-\Pi_1(v)\|_{L^2(K_R)}\leq
C\sqrt{1-\frac{H}{\kappa}}\,\|v\|_{L^2(K_R)}$, we can replace
$\|\Pi_1(v)\|_{L^2(K_R)}$ by $\|v\|_{L^2(K_R)}$ to leading order.
That way we get,
$$
\frac1{|Q_\ell|}\int_{K_R}\Big(\left(1-\frac{\kappa}{H}\right)\big(1+o(1)\big)|v|^2+\frac{\kappa}{2H}|v|^4\Big)\,dx\leq[\kappa-H]^2\Big(E_{\rm
Ab}+o(1)\Big)\,.
$$
Applying the change of variable $x\mapsto a_j+\frac{x}{\sqrt{\kappa
H}}$ and remembering the definition of $v$ in Theorem~\ref{thm:Al}
we get,
$$
\frac{\kappa H}{|Q_\ell|}\int_{Q_\ell}\Big(\left(1-\frac{\kappa}{H}\right)\big(1+o(1)\big)|\chi_\ell\psi|^2+\frac{\kappa}{2H}|\chi_\ell\psi|^4\Big)\,dx\leq[\kappa-H]^2\Big(E_{\rm
Ab}+o(1)\Big)\,.
$$
Theorem~\ref{thm:loc-en} tells us that
$\displaystyle\frac1{|Q_\ell|}\int_{Q_\ell}|\chi_\ell\psi|^4\,dx=-2E_{\rm
Ab}\left[1-\frac{H}{\kappa}\right]^2+o\left(\left[1-\frac{H}{\kappa}\right]^2\right)$.
Consequently, we get that,
$$
\frac{\kappa H}{|Q_\ell|}\int_{Q_\ell}\left(1-\frac{\kappa}{H}\right)\big(1+o(1)\big)|\chi_\ell\psi|^2\,dx\leq
E_{\rm
Ab}\left[\kappa-H\right]^2+[\kappa-H]^2\Big(E_{\rm
Ab}+o(1)\Big)\,.
$$
Remembering the assumptions \eqref{eq-as-H1}-\eqref{eq-as-H2} on
$H$, we deduce that,
\begin{equation}\label{eq:R-lb}
\frac{1}{|Q_\ell|}\int_{Q_\ell}\big(1+o(1)\big)|\chi_\ell\psi|^2\,dx\geq
-2E_{\rm
Ab}\left[1-\frac{H}{\kappa}\right]+o\left(\left[1-\frac{H}{\kappa}\right]\right)\,.
\end{equation}
Now we establish a matching upper bound.
We introduce the  parameters
$$\alpha=\left(1-\frac{H}\kappa\right)^{1/16}\,,\quad \epsilon=\left(1-\frac{H}\kappa\right)^{3/8}\,,
\quad \ell'=(\kappa-H)^{-1}\epsilon\quad{\rm and}\quad R'=\ell'\sqrt{\kappa H}\,.
$$
These parameters satisfy
$$
\left(1-\frac{H}{\kappa}\right)^2 R'^2\ll 1\,,\quad \kappa^{-1}\ll\ell'\ll \ell\,,\quad 1\ll R'\ll R\quad{\rm and}~(\ell')^{-2}\alpha^{-2}\left(1-\frac{H}\kappa\right)\ll [\kappa-H]^2\,.
$$
We cover the square $Q_\ell$ by  $N$ pariwise dsjoint squares $(\widetilde{Q}_{\ell',i})_i$ of side length $\ell'$.
These squares are constructed as follows. Then we replace every square $\widetilde{Q}_{\ell',i}$ by ${Q}_{\ell',i}$ with the same center but a slightly larger side-length $(1+\alpha)\ell'$ (see Figure~\ref{fig}).
\begin{figure} \label{fig} 
 \centering
 \includegraphics[scale=0.7]{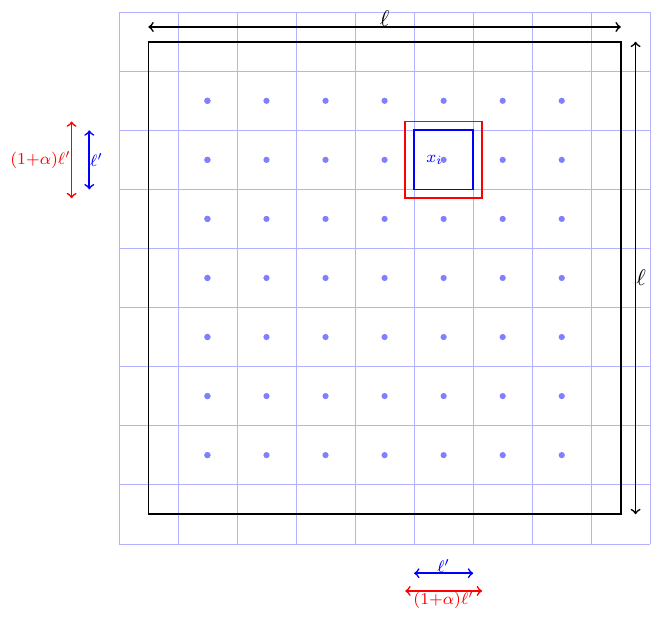}
\caption{The  square $Q_{\ell}$ decomposed into the small squares $\widetilde Q_{\ell',i}$. Note the representation of the square $\widetilde Q_{\ell',i}$ with center $x_i$ and the slightly larger square $Q_{\ell',i}$.}
\end{figure} 
The number $N$ satisfies
\begin{equation}\label{eq:Nb*}
\left| N-\frac{\ell^2}{\ell'^2}\right|\leq C\frac{\ell}{(\ell')^2}\,.
\end{equation}  
Consider a partition of unity $(g_i)$  satisfying in $Q_\ell$
$$\sum_{i}g_i=1\,,\quad \sum_i|\nabla g_i|^2\leq C(\ell')^{-2}\alpha^{-2}\,,\quad {\rm supp}~g_{i}\subset Q_{\ell',i}\,.$$
We have,
$$\begin{aligned}
0\geq E_{\rm Ab}[\kappa-H]^2+o([\kappa-H]^2)&\geq\mathcal{E}_{0}(\chi_\ell\psi,\Ab_0;Q_{\ell})\\
&\geq  \sum_{i}\mathcal E_0(g_i \chi_\ell\psi,\Ab_0;Q_{\ell',i})-\ell^2[\kappa-H]^2o(1)\,.
\end{aligned}$$
%
%
%
Let $N_+={\rm Card},\mathcal J_+$ and $N_-={\rm Card}\,\mathcal J_-$, where
$$\mathcal J_+=\{i~:~q(g_i\chi_\ell \psi,\Ab_0;Q_{\ell',i})> 0\}\,,\quad
\mathcal J_-=\{i~:~q(g_i\chi_\ell \psi,\Ab_0;Q_{\ell',i})\leq 0\}\,,$$
and
$$q(g_i\chi_\ell \psi,\Ab_0;Q_{\ell',i})=\int_{Q_{\ell',i}}\Big(|(\nabla-i\kappa H\Ab_0)h_{i}v|^2-\kappa^2|h_{i}v|^2\Big)\,dx\,.$$
 We have
$$
 \left| N_- -\frac{\ell^2\sqrt{\kappa H}}{\ell'^2}\right| \leq \frac{\ell^2\sqrt{\kappa H}}{\ell'^2}\,o(1)\,,\quad
and\quad N_+=N_-o(1)\,.
$$
Let $x_i$ denote the center of the square $Q_{\ell',i}$, $R'=\ell\sqrt{\kappa H}$ and $K_{R'}=(-R'/2,R'/2)^2$. We introduce the two functions
$$
h_i(x)=g\Big(x_j+\frac{x}{\sqrt{\kappa H}}\Big)\quad {\rm and}\quad v(x)=\left(\chi_\ell\psi\right)\Big(x_i+\frac{x}{\sqrt{\kappa H}}\Big)\,,\quad \big(\,x\in
K_R'\,\big)
$$
As in the proof of Theorem~\ref{thm:Al}, for all $i\in\mathcal J_-$, we have
\begin{align*}
\mathcal E_0(g_i\chi_\ell\psi,\Ab_0;Q_{\ell',i})&\geq\int_{K_R'}\Big(\left(1-\frac{\kappa}{H}\right)|\Pi_1(h_iv)|^2+\frac{\kappa}{2H}|h_iv|^4\Big)\,dx\\
&\geq \left[1-\frac{\kappa}{H}\right]_+^2\bigg(\left(1+2\left(\sigma+1-\frac{\kappa}{H}\right)\right)c(R')-C\sigma^{-3}\left(1-\frac{H}{\kappa}\right)^2(R')^4\bigg)\,,
\end{align*}
and
$\|h_iv-\Pi_1(h_iv)\|_{L^2(K_{R'})}\ll \|h_iv\|_{L^2(K_{R'})}$\,,
for all $\sigma\in(0,1/2)$. Note that our choice of $R'$ allows us to choose $\sigma\ll1$ such that $\sigma^{-3}\left(1-\frac{H}{\kappa}\right)^2(R')^4\ll (R')^2$. Thus, we get, for all $i\in\mathcal J_-$,
\begin{align*}
\mathcal E_0(g_i\chi_\ell \psi,\Ab_0;Q_{\ell',i})&\geq\int_{K_{R'}}\Big(\left(1-\frac{\kappa}{H}\right)|\Pi_1(h_iv)|^2+\frac{\kappa}{2H}|h_iv|^4\Big)\,dx\\
&\geq \left[1-\frac{\kappa}{H}\right]_+^2\bigg(c(R')+o\big((R')^2\big)\bigg)\,.
\end{align*}
We replace $\|\Pi(h_iv)\|_2$ by $\|h_iv\|_2$ and sum over $i\in\mathcal J_-$ to get,
$$\sum_{i\in \mathcal J_-}\int_{K_{R'}}\Big(\left(1-\frac{\kappa}{H}\right)|h_iv|^2+\frac{\kappa}{2H}|h_iv|^4\Big)\,dx\\
\geq N_-\left[1-\frac{\kappa}{H}\right]_+^2\bigg(c(R')+o\big((R')^2\big)\bigg)\,.
$$
Since $h_i^2\geq h_i^4$, $c(R')=(R')^2E_{\rm Ab}+o((R')^2)$, $N_-=N+o(N)$ and $N$ satisfies \eqref{eq:Nb*}, we get
$$
\sum_{i\in \mathcal J_-}\int_{K_{R'}}\Big(\left(1-\frac{\kappa}{H}\right)|h_iv|^2+\frac{\kappa}{2H}h_i^2|v|^4\Big)\,dx\\
\geq \ell^2\kappa H\sqrt{\kappa H}\left[1-\frac{\kappa}{H}\right]_+^2E_{\rm Ab}\bigg(1+o(1)\bigg)\,.
$$
Using the bound $|v|\leq C(1-\frac{H}{\kappa})^{1/2}$ and that the number $N_+$ of indices in $\mathcal J_+$ is equal to $o(N)$, we get that the sum over $i\in\mathcal J$ satisfies,
$$
\sum_{i\in \mathcal J}\int_{K_{R'}}\Big(\left(1-\frac{\kappa}{H}\right)|h_iv|^2+\frac{\kappa}{2H}h_i^2|v|^4\Big)\,dx\\
\geq \ell^2\kappa H\sqrt{\kappa H}\left[1-\frac{\kappa}{H}\right]_+^2E_{\rm Ab}\bigg(1+o(1)\bigg)\,.
$$
Now,  noting that $\sum_{i\in\mathcal J}h_i^2=1$ and performing a change of variable, we get 
$$\int_{Q_\ell}\Big(\left(1-\frac{\kappa}{H}\right)|\chi_\ell \psi|^2+\frac{\kappa}{2H}|\chi_\ell \psi|^4\Big)\,dx\\
\geq \ell^2\kappa H\sqrt{\kappa H}\left[1-\frac{\kappa}{H}\right]_+^2E_{\rm Ab}\bigg(1+o(1)\bigg)\,,$$
Using the asymptotics for $\|\chi_\ell\psi\|_4$ and a change of variables yields the following upper bound
\begin{equation}\label{eq:R-ub}
\frac{1}{|Q_\ell|}\int_{Q_\ell}|\chi_\ell\psi|^2\,dx\leq
-2E_{\rm
Ab}\left[1-\frac{H}{\kappa}\right]+o\left(\left[1-\frac{H}{\kappa}\right]\right)\,.
\end{equation}
We collect \eqref{eq:R-lb} and \eqref{eq:R-ub}, then we use 
the assumption on the support of the function $\chi_\ell$, the bound
\eqref{eq-psi<1-b} and the choice of $\ell$ to get,
$$\int_{Q_\ell}|\psi|^2\,dx=\int_{Q_\ell}|\chi_\ell\psi|^2\,dx+\mathcal
O\left(\frac{\ell}{\sqrt{\kappa H}}\left[1-\frac{H}{\kappa}\right]\right)=E_{\rm Ab}\ell^2\left[1-\frac{H}\kappa\right]+o\left(\ell^2\left[1-\frac{H}{\kappa}\right]\right)\,.$$
 This
finishes the proof of Theorem~\ref{thm:main}.
\end{proof}

\section*{Acknowledgements}
The author would like to thank K. Attar for his reading of the first
version of this paper. The research of the author is supported by a
 grant from Lebanese University.

\end{document}